\documentclass{gtpart}

\usepackage{graphicx}
\usepackage{amsmath, amsfonts, mathrsfs, pinlabel}

\title{Algebraic degrees of stretch factors in mapping class groups}

\author{Hyunshik Shin}

\givenname{Hyunshik}
\surname{Shin}

\address{Department of Mathematical Sciences, KAIST\\\newline
         291 Daehak-ro Yuseong-gu Daejeon 34141 South Korea}
\email{hshin@kaist.ac.kr}
\urladdr{http://mathsci.kaist.ac.kr/~hshin}


\arxivreference{1401.1836}  
\arxivpassword{}   

\newcommand{\nc}{\newcommand}
\nc{\dmo}{\DeclareMathOperator}
\nc{\nt}{\newtheorem}

\nt{theorem}{Theorem}
\nt{lemma}[theorem]{Lemma}
\nt{prop}[theorem]{Proposition}
\nt{question}[theorem]{Question}
\nt{cor}[theorem]{Corollary}
\nt{problem}[theorem]{Problem}
\nt*{conj}{Conjecture}
\nt{remark}[theorem]{Remark}

\newtheorem{thm}{Theorem}

\nt*{ques}{Question}
\nt*{proposition}{Proposition}
\nt*{rem}{Remark}

\nc{\N}{\mathbb{N}}
\dmo{\MCG}{Mod}
\dmo{\PMF}{\mathcal{PMF}}
\dmo{\PSL}{PSL}
\dmo{\ra}{\rightarrow}
\dmo{\tr}{tr}
\nc{\F}{\mathcal{F}}
\nc{\st}{\lambda}
\nc{\hst}{\lambda_{H}}
\nc{\A}{\mathcal{A}}
\nc{\E}{\mathcal{E}}
\nc{\G}{\mathcal{G}}

\nc{\margin}[1]{\marginpar{\tiny #1}}
\nc{\p}[1]{\smallskip\noindent{{\bf #1}}}

\begin{document}

\begin{abstract}
We explicitly construct pseudo-Anosov maps on the closed surface of genus $g$ with orientable foliations whose stretch factor $\lambda$ is a Salem number with algebraic degree $2g$. Using this result, we show that there is a pseudo-Anosov map whose stretch factor has algebraic degree $d$, for each positive even integer $d$ such that $d \leq g$.
\end{abstract}

\maketitle

\section{Introduction}
\label{section:introduction}
Let $S_g$ be a closed surface of genus $g \geq 2$. The \textit{mapping class group} of $S_g$, denoted $\MCG(S_g)$, is the group of isotopy classes of orientation preserving homeomorphisms of $S_g$. An element $f \in \MCG (S_g)$ is called a \textit{pseudo-Anosov} mapping class if there are transverse measured foliations $(\F^u,\mu_u)$ and $(\F^s,\mu_s)$, a number $\lambda(f) >1$, and a representative homeomorphism $\phi$ such that
\[
 \phi(\F^u,\mu_u) = (\F^u,\lambda(f) \mu_u) \ \ \textrm{and} \ \ \phi(\F^s,\mu_s) = (\F^s,\lambda(f)^{-1} \mu_s).
\]
In other words, $\phi$ stretches along one foliation by $\lambda(f)$ and the other by $\lambda(f)^{-1}$. The number $\lambda(f)$ is called the \textit{stretch factor} (or \textit{dilatation}) of $f$.

A pseudo-Anosov mapping class is said to be orientable if its invariant foliations are orientable. Let $\hst(f)$ be the spectral radius of the action of $f$ on $H_1(S_g;\R)$. Then
\[
\hst(f) \leq \lambda(f),
\]
and the equality holds if and only if the invariant foliations for $f$ are orientable (see \cite{minimal}). The number $\lambda_{H}(f)$ is called the \textit{homological stretch factor} of $f$.

\begin{ques}
Which real numbers can be stretch factors?
\end{ques}

It is a long-standing open question. Fried \cite{fried} conjectured that $\lambda>1$ is a stretch factor if and only if all conjugate roots of $\lambda$ and $1/\lambda$ are strictly greater than $1/\lambda$ and strictly less than $\lambda$ in magnitude.

Thurston \cite{thurston} showed that a stretch factor $\lambda$ is an algebraic integer whose algebraic degree has an upper bound $6g-6$. More specifically, $\lambda$ is the largest root in absolute value of a monic palindromic polynomial. Thurston gave a construction of mapping classes of $\MCG(S_g)$ generated by two multitwists and he mentioned that his construction can make a pseudo-Anosov mapping class whose stretch factor has algebraic degree $6g-6$. However, he did not give specific examples. 

What happens if we fix the genus $g$? To simplify the question, we may ask which algebraic degrees are possible on $S_g$.

\begin{ques}
What degrees of stretch factors can occur on $S_g$?
\end{ques}

Very little is known about this question. Using Thurston's construction, it is easy to find quadratic integers as stretch factors. Neuwirth and Patterson \cite{neuwirthpatterson} found non-quadratic examples, which are algebraic integers of degree 4 and 6 on surfaces of genus 4 and 6, respectively. Using interval exchange maps, Arnoux and Yoccoz \cite{arnouxyoccoz} gave the first generic construction of pseudo-Anosov maps whose stretch factor has algebraic degree $g$ on $S_g$ for each $g \geq 2$.

\subsection*{Main Theorems}
In this paper, we give a generic construction of pseudo-Anosov mapping classes with stretch factor of algebraic degree $2g$.

Let $c_i$ and $d_j$ be simple closed curves on $S_g$ as in Figure \ref{figure:pA}.
For $k \geq 3$, let us define
\[
f_{g,k} =T_{A_{g,k}} T_{B_g},
\]
where $T_{A_{g,k}} = \left( T_{c_1} T_{c_2} \cdots T_{c_{g-1}} \right) \left( T_{c_g} \right)^k$ and $T_{B_g} = T_{d_1} \cdots T_{d_g}$.
Here, $T_{\alpha}$ is the Dehn twist about $\alpha$. We will show that $f_{g,k}$ is a pseudo-Anosov mapping class and its stretch factor $\lambda(f_{g,k})$ is a special algebraic integer, called Salem number. 
A \textit{Salem number} is an algebraic integer $\alpha>1$ whose Galois conjugates other than $\alpha$ have absolute value less than or equal to 1 and at least one conjugate lies on the unit circle.

\begin{thm}
\label{thm:mainA}
For each $g \geq 2$ and $k \geq 3$, $f_{g,k}$ is a pseudo-Anosov mapping class and satisfies the following properties:
\begin{enumerate}
\item $\st ( f_{g,k} ) = \hst ( f_{g,k} )$,
\item $\st ( f_{g,k} )$ is a Salem number, and
\item $\lim\limits_{g \rightarrow \infty} \lambda( f_{g,k} ) = k-1$.
\end{enumerate}
\end{thm}

\begin{figure}[tp]
\labellist \small\hair 2pt
\pinlabel $S_g$ at 20 150
\pinlabel $d_1$ at 362 102
\pinlabel $c_1$ at 325 97
\pinlabel $d_2$ at 286 102
\pinlabel $c_2$ at 250 97
\pinlabel $d_g$ at 61 102
\pinlabel $c_g$ at 21 97
\endlabellist
\centering
\includegraphics[scale=0.65]{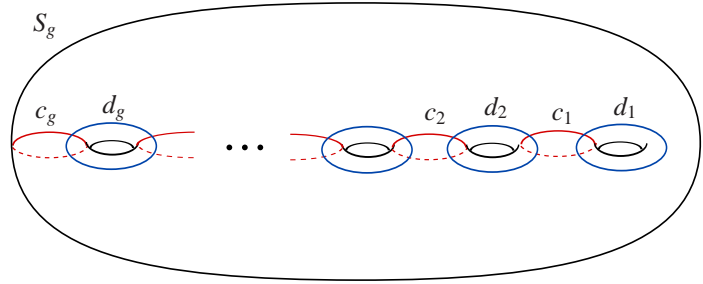}
\caption{Simple closed curves on $S_g$}
\label{figure:pA}
\end{figure}

In particular, we will prove that for $k=4$, the algebraic degree of stretch factor  
is $2g$.
It is known that the degree of the stretch factor of a pseudo-Anosov mapping class 
$f \in \MCG(S_g)$ with orientable foliations is bounded above by $2g$ 
(see \cite{thurston}). Therefore our examples give the maximum degrees of stretch factors
for orientable foliations in $\MCG(S_g)$ for each $g \geq 2$.

\begin{thm}
\label{thm:mainB}
Let $f_g \in \MCG(S_g)$ be the mapping class given by
$$f_g = f_{g,4} = T_{A_{g,4}} T_{B_g}.$$
Then the minimal polynomial of the stretch factor $\lambda(f_g)$ is
$$ p_g(x) = x^{2g}  - 2 \left( \sum_{j=1}^{2g-1} x^j \right) + 1. $$
This implies
$$\deg \st(f_g) = 2g.$$
\end{thm}

\medskip
The hard part is to show the irreducibility of $p_g(x)$, which is proved in section \ref{section:irreducibility}.

In general, for each $k \geq 3$, the Salem stretch factor of $f_{g,k}$
is the root of the polynomial
\[
p_{g,k}(x) = x^{2g}  - (k-2) \left( \sum_{j=1}^{2g-1} x^j \right) + 1.
\]
It can be shown that $p_{g,k}(x)$ is irreducible for each $k \geq 4$, 
but since the main purpose of this paper is degree 
realization, we will prove only for $k=4$ case that 
the algebraic degree of the stretch factor is $2g$.

Using a branched cover construction, we use Theorem \ref{thm:mainB} to deduce the following partial answer to our question about algebraic degrees.	

\newtheorem*{thm:even_degree}{Corollary \ref{thm:even_degree}}
\begin{thm:even_degree}
For each positive integer $h \leq g/2$, there is a pseudo-Anosov mapping class $\widetilde{f_h} \in \MCG(S_g)$ such that $\deg(\lambda(\widetilde{f_h}))=2h$ and $\st(\widetilde{f_h})$ is a Salem number.
\end{thm:even_degree}

\subsection*{Obstructions.}
 There are three known obstructions for the existence of algebraic degrees. For any pseudo-Anosov $f \in \MCG(S_g)$, we have:
\begin{enumerate}
\item $\deg \lambda(f) \geq 2$,
\item $\deg \lambda(f) \leq 6g-6$, and
\item if $\deg \lambda(f) > 3g-3$, then $\deg \lambda(f)$ is even.
\end{enumerate}

The third obstruction is due to Long \cite{long} and we have another proof in section \ref{section:odd}. It turns out these are the only obstructions for $g=2$. However it is not known whether there are other obstructions of algebraic degrees for $g \geq 3$. By computer search, odd degree stretch factors are rare compared to even degrees. We conjecture that every even degree $d \leq 6g-6$ can be realized as the algebraic degree of stretch factors.
\begin{conj}
On $S_g$, there exists a pseudo-Anosov mapping class with a stretch factor of algebraic degree $d$ for each positive even integer $d \leq 6g-6$.
\end{conj}

In section \ref{section:examples}, we show that the conjecture is true for $g = 2,3,4,$ and $5$.

\subsection*{Outline}
In section \ref{section:background} we will give the basic definitions and results about 
Thurston's consturction.
We will prove 
Theorem \ref{thm:mainA} in section \ref{section:proofbycoxeter} by the theory of Coxeter graphs.
In section \ref{section:covers}, we construct pseudo-Anosov mapping classes via branched covers.
In section \ref{section:odd}, we explain some properties of odd degree stretch factors. 
Section \ref{section:examples} contains examples of even degree stretch factors for $g=2,3,4$ and 5.
Section \ref{section:irreducibility} is where we prove Theorem \ref{thm:mainB}, that is, 
we prove that the minimal polynomial of $\lambda(f_g)$ has degree $2g$.

\subsection*{Acknowledgments}
I am very grateful to my advisor Dan Margalit for numerous help and discussions. I would also like to thank Joan Birman, Benson Farb, Daniel Groves, Chris Judge, and Bal\'azs Strenner for helpful suggestions and comments. I wish to thank an anonymous referee for very helpful comments. Lastly, I would like to thank the School of Mathematics of Georgia Institute of Technology for their hospitality during the time in which the major part of this paper was made.

\medskip
\section{Background}
\label{section:background}

\subsection*{Thurston's construction}
We recall Thurston's construction of mapping classes \cite{thurston}.
For more details on this material, see \cite{primer} or \cite{Leininger04}.

Suppose $A = \{ a_1, \ldots, a_n \}$ is a set of pairwise disjoint simple closed curves,
called a \textit{multicurve}.
We denote the product of Dehn twists $\prod_{i=1}^{n} T_{a_i}$ by $T_A$.
This product is called a \textit{multitwist}.

Suppose $A= \{ a_1, \ldots, a_n \}$ and $B= \{ b_1, \ldots, b_m \}$ are multicurves in a surface $S$
 so that $A \cup B$ \textit{fills} $S$, that is,
the complement of $A \cup B$ is a disjoint union of disks and once-punctured disks.
Let $N$ be the $n \times m$ matrix whose $(j,k)$-entry is the geometric intersection number
$i(a_j, b_k)$ of $a_j$ and $b_k$. Let $\nu = \nu(A \cup B)$ be the largest eigenvalue in magnitude of 
the matrix $NN^t$. If $A \cup B$ is connected, then $NN^t$ is primitive and by
the Perron--Frobenius theorem $\nu$ is a positive real number greater than 1
(see \cite[p. 392 - 395]{primer} for more detail).

Thurston constructed a singular Euclidean structure on $S$ with respect to
which $\langle T_A, T_B \rangle$ acts by affine transformations given by
the representation $\rho:\langle T_A, T_B \rangle \rightarrow \PSL(2,\R)$
\[
\rho(T_A) = 
\left( \begin{array}{cc}
1 & -\nu^{1/2}\\
0 & 1
\end{array} \right)
\quad
\textrm{and}
\quad
\rho(T_B) =
\left( \begin{array}{cc}
1 & 0\\
\nu^{1/2} & 1
\end{array} \right).
\]
In particular, an element $f \in \langle T_A, T_B \rangle$ is pseudo-Anosov
if and only if $\rho(f)$ is a hyperbolic element in $\PSL(2,\R)$
and then the stretch factor $\lambda(f)$ is 
equal to the bigger eigenvalue of $\rho(f)$. For instance,
for a mapping class $f = T_A T_B$,
\[
\rho(T_A T_B) = 
\left( \begin{array}{cc}
1 & -\nu^{1/2}\\
0 & 1
\end{array} \right)
\left( \begin{array}{cc}
1 & 0\\
\nu^{1/2} & 1
\end{array} \right)
=
\left( \begin{array}{cc}
1-\nu & -\nu^{1/2}\\
\nu^{1/2} & 1
\end{array} \right),
\]
and the stretch factor $\lambda(T_A T_B)$ is the bigger root of the characteristic polynomial
\[
\lambda^2 - \lambda (\nu-2) + 1,
\]
provided that $\nu -2 > 2$.

\medskip

\section{Proof by the theory of Coxeter graphs}
\label{section:proofbycoxeter}
We will prove Theorem \ref{thm:mainA} in this section.

For the set $C$ of simple closed curves on the surface $S_g$, 
the \textit{configuration graph} for $C$, denoted $\G(C)$, is the graph
with a vertex for each simple closed curve and an edge for every point of intersection between simple
closed curves.

Let $f_{g,k}$ be a mapping class on $S_g$ defined by
\[
f_{g,k} = T_{A_{g,k}} T_{B_g}, \	 k \geq 3,
\]
as in Theorem \ref{thm:mainA}.
By regarding the multiple power of $T_{c_g}$ as the product of Dehn twists about parallel
(isotopic) simple closed curves $c_{g_1}, \ldots, c_{g_{k}}$,
let us define the multicurves
\begin{equation*}
A_{g,k} = \{ c_1, \ldots, c_{g-1}, c_{g_1}, \ldots, c_{g_{k}} \} \ \ \textrm{and}
\ \ B_g = \{ d_1, \ldots, d_g \}.
\end{equation*}
Then the configuration graph $\G(A_{g,k} \cup B_g)$ is a tree as in Figure \ref{figure:configurationgraph},
\begin{figure}[t]
\centering
\includegraphics[scale=0.7]{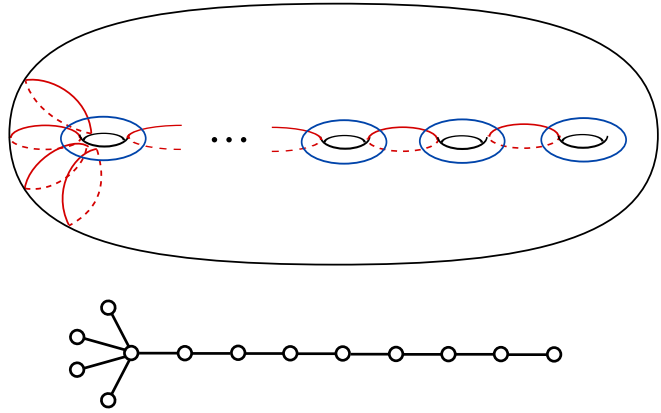}
\caption{Multicurves and configuration graph $\G(A_{g,k} \cup B_g)$}
\label{figure:configurationgraph}
\end{figure}

\medskip

\subsection{Coxeter graphs and mapping class groups}
We say that a finite graph $\G$ is a \textit{Coxeter graph} if there are no self-loops or multiple edges.
For given multicurves $A$ and $B$ such that $A \cup B$ fills the surface $S$,
suppose that the configuration graph $\G = \G(A\cup B)$ is a Coxeter graph.
Leininger proved the following theorem.

\begin{theorem}[\cite{Leininger04} Theorem 8.1 and Theorem 8.4]
Let $\G(A \cup B)$ be a  non-critical dominant Coxeter graph.
Then $T_A T_B$ is a 
pseudo-Anosov mapping class with stretch factor $\lambda$ such that
\[
\lambda^2 + \lambda(2-\mu^2) + 1 = 0,
\]
where $\mu$ is the spectral radius of the graph $\G$.
\label{thm:Leininger}
\end{theorem}

For the definitions and pictures of critical and dominant graphs, see \cite[Section 1]{Leininger04}

For the multicurves $A_{g,k}$ and $B_g$ in Theorem \ref{thm:mainA},
$\G(A_{g,k} \cup B_g)$ is a non-critical dominant Coxeter graph for each $k \geq 3$.
Therefore by Theorem \ref{thm:Leininger} the mapping class $f_{g,k} = T_A T_B$ is pseudo-Anosov for each $k \geq 3$.

\medskip
\subsection{Orientability}
Suppose that $\G$ is a connected Coxeter graph with the set $\Sigma$ of vertices.
There is an associated quadratic form $\Pi_{\G}$ on $R^{\Sigma}$ and a faithful representation
\[
\Theta: \mathscr{C}(\G) \ra \textrm{O}(\Pi_{\G}),
\]
where $\mathscr{C}(\G)$ is a Coxeter group with generating set $\Sigma$,
$\textrm{O}(\Pi_{\G})$ is the orthogonal group of the quadratic form $\Pi_{\G}$, and
each generator $s_i \in \Sigma$ is represented by a reflection.
Leininger also proved the following theorem.

\begin{theorem}[\cite{Leininger04} Theorem 8.2 ]
\label{thm:Leininger_homology}
Let $\G(A \cup B)$ be a Coxeter graph and suppose that $A$ and $B$ can be oriented so that all
intersections of $A$ with $B$ are positive. Then there exists a homomorphism
\[
\eta: \R^{\Sigma} \ra H_1(S;\R)
\]
such that 
\[
(T_A T_B)_{\ast} \circ \eta = -\eta \circ \Theta(\sigma_A \sigma_B),
\]
where $\sigma_A \sigma_B$ is an element in $\mathscr{C}(\G)$ corresponding to $T_A T_B$.\\
Moreover, $\Theta(\sigma_A \sigma_B)|_{\textrm{ker}(\eta)} = -I$ and $\eta$ preserves spectral radii.
\end{theorem}

\begin{figure}[htb]
\centering
\includegraphics[scale=0.7]{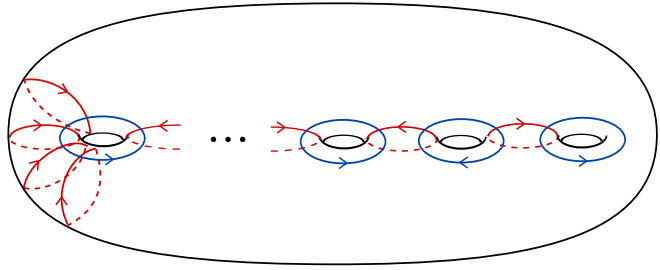}
\caption{Orientation of positive intersections}
\label{figure:positiveorientation}
\end{figure}

Theorem \ref{thm:Leininger_homology} implies that if $A$ and $B$ can be oriented as in the theorem,
then the stretch factor of a pseudo-Anosov mapping class is equal to the spectral radius of the action on homology.
For multicurves $A_{g,k}$ and $B_g$ in Theorem \ref{thm:mainA}, they can be oriented so that all intersections are positive
as in Figure \ref{figure:positiveorientation}. Therefore we have
\[
\st ( f_{g,k} ) = \hst ( f_{g,k} )
\]
and the invariant foliations for $f_{g,k}$ are orientable.

It is also possible to directly compute the action on the first homology. 
Consider the mapping class $f_g = T_{A_{g,4}} T_{B_g}$ as in Theorem \ref{thm:mainB}.
Let us choose a basis $\{ a_1, b_1, \dots, a_g, b_g \}$ for $H_1(S_g)$ 
as in Figure \ref{figure:homology}.

\begin{figure}[htb]
\labellist \small\hair 2pt
\pinlabel $S_g$ at 20 150
\pinlabel $a_1$ at 362 100
\pinlabel $b_1$ at 325 100
\pinlabel $a_2$ at 286 100
\pinlabel $b_2$ at 250 100
\pinlabel $a_3$ at 210 100
\pinlabel $a_g$ at 61 100
\pinlabel $b_g$ at 21 100
\endlabellist
\centering
\includegraphics[scale=0.70]{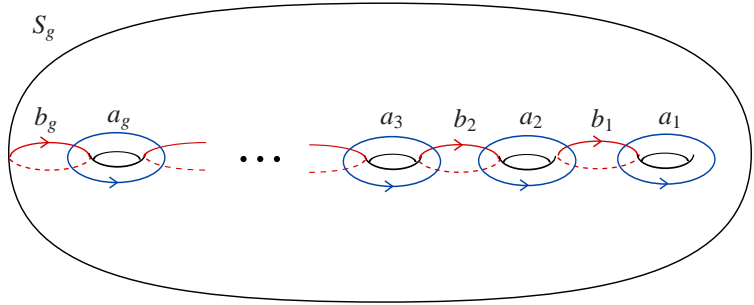}
\caption{A basis for $H_1(S_g)$.}
\label{figure:homology}
\end{figure}

By computing images of each basis element under $f_g$, we can get the action on $H_1(S_g)$
\[
\left( \begin{array}{rrrrcr}
1 & -1 & 0 & ~0 & \cdots &  0\\
1 & 0 & -1 & 0 & \cdots &  0\\
\vdots & \vdots & \vdots & \vdots & \ddots &  \vdots\\
1 & 0 & 0 & 0 & \cdots &  -1\\
4 & 0 & 0 & 0 & \cdots &  -3
\end{array} \right).
\]
By induction, the characteristic polynomial $h_g(x)$ of the homological action is
\[
h_g(x) = x^{2g}  + 2\left( \sum_{j=1}^{2g-1} (-1)^j x^j \right) + 1.
\]
Since the largest root of $h_g(x)$ in magnitude is a negative real number, we can deduce that the stretch factor 
$\st(f_g)$ is the root of $h_g(-x)$. Specifically, $\st(f_g)$ is the root of
$$ p_g(x) = x^{2g}  - 2 \left( \sum_{j=1}^{2g-1} x^j \right) + 1. $$

In a similar way, one can get the polynomial for $\st(f_{g,k})$, which is
$$ p_{g,k}(x) = x^{2g}  - (k-2) \left( \sum_{j=1}^{2g-1} x^j \right) + 1. $$


\medskip
\subsection{Salem numbers and spectral properties of starlike trees}
The configuration graph $\G(A_{g,k} \cup B_g)$ for $f_{g,k}$ is a special type of graphs, called a starlike
tree, and its relation to Salem numbers is studied in \cite{McKeeRowlinsonSmyth99}.
A \textit{starlike tree} is a tree 	with at most one vertex of degree $>2$.
Let $T=T(n_1, n_2, \ldots, n_k)$ be the starlike tree with $k$ arms of
$n_1,n_2,\ldots,n_k$ edges.

\begin{theorem}[\cite{McKeeRowlinsonSmyth99} Corollary 9]
Let $T = T(n_1, n_2, \ldots, n_k)$ be a starlike tree and let $\mu$ be the spectral radius of $T$.
Suppose that $\mu$ is not an integer and $T$ is a non-critical dominant graph.
Then $\lambda >1$, defined by $\sqrt{\lambda} + 1/\sqrt{\lambda} = \mu$, is a Salem number.
\label{thm:starlikegraph}
\end{theorem}

The configuration graph
$\G(A_{g,k} \cup B_g)$ in Theorem \ref{thm:mainA} is a non-critical dominant starlike tree 
$$ T(2g-2, \underbrace{1, 1, \ldots, 1}_{k\textrm{-times}}), \ k \geq 3$$
and we will denote it by $T(2g-2, k \cdot 1$). 
The fact that the spectral radius of $T(2g-2, k \cdot 1)$ is not an integer follows from the following
theorem.

\begin{theorem}[\cite{Shin15}]
If $\mu$ is the spectral radius of the starlike tree $T(n, k \cdot 1)$, then
\[
\sqrt{k+1} < \mu < \frac{k}{\sqrt{k-1}}
\]
for $n \geq 1$ and $k \geq 3$.
\end{theorem}

Thus for the starlike tree $T(n, k \cdot 1)$, the spectral radius satisfies
$$ k+1 < \mu^2 <\frac{k^2}{k-1} = k+1 +\frac{1}{k-1}.$$
Therefore $\mu$ is not an integer and by Theorem \ref{thm:starlikegraph} $\st(f_{g,k})$ is a Salem number.

Moreover, the proof of Corollary 2.1 of Lepovi\'c--Gutman \cite{LepovicGutman01}
implies that
$$\lim\limits_{g \rightarrow \infty} \lambda( f_{g,k} ) = k-1.$$
For completeness, we reprove this here.

Recall that $\st(f_{g,k})$ is the largest root of 
\begin{equation*}
p_{g,k}(x) = x^{2g}  - (k-2) \left( \sum_{j=1}^{2g-1} x^j \right) + 1.
\label{eqn:charpoly}
\end{equation*}
By multiplying $p_{g,k}(x)$ by $x-1$, the stretch factor $\st(f_{g,k})$
is the largest root in magnitude of
$$q_{g,k}(x) = x^{2g+1} - (k-1) x^{2g} + (k-1) x -1.$$
We have $ q_{g,k}(k-1) = (k-1)^2 - 1 >0$, and for any fixed positive integer $m$,
{\setlength\arraycolsep{2pt}
\begin{eqnarray*}
q_{g,k}\left( k-1-\frac{1}{10^m} \right) &=& \left( k-1-\frac{1}{10^m} \right)^{2g} \left(-\frac{1}{10^m} \right) + (k-1) \left( k-1-\frac{1}{10^m} \right) - 1
\end{eqnarray*}}
Hence $q_{g,k}\left(k-1-10^{-m}\right) < 0$ for sufficiently large values of $g$
and therefore $p_{g,k}(x)$ has a root on the interval $(k-1-10^{-m}, k-1)$. This implies
$$\lim\limits_{g \rightarrow \infty} \lambda( f_{g,k} ) = k-1.$$
This completes the proof of Theorem \ref{thm:mainA}.

\begin{rem}
\emph{A positive integer cannot be a stretch factor (which is an algebraic integer of degree 1). 
However, Theorem \ref{thm:mainA} implies that for sufficiently large genus $g$ there is a stretch factor which is a Salem number arbitrarily close to 
a given integer $k-1$ for each $k \geq 3$.
}
\end{rem}

\medskip
\section{Branched Covers}
\label{section:covers}
Lifting a pseudo-Anosov mapping class via a covering map is one way to construct another pseudo-Anosov mapping class. If there is a branched cover $\widetilde{S} \rightarrow S$ and a pseudo-Anosov mapping class $f \in \MCG(S)$, then there is some $k \in \N$ such that $\MCG(\widetilde{S})$ has a pseudo-Anosov element $\widetilde{f}$ which is a lift of $f^k$ and hence $\lambda(\widetilde{f}) = \lambda(f)^k$.

\begin{cor}
\label{thm:even_degree}
Let $ g \geq 2$. For each positive integer $h \leq g/2$, there is a pseudo-Anosov mapping class $\widetilde{f_h} \in \MCG(S_g)$ such that $\deg(\lambda(\widetilde{f_h}))=2h$ and $\st(\widetilde{f_h})$ is a Salem number.
\end{cor}

\begin{proof}
Let
\[
h = \left\{ \begin{array}{ll}
			\frac{g-2m}{2}, & \textrm{if $g$ is even, $m=0, 1, \ldots, (g-2)/2$,}\\
			&\\
			\frac{g-1 -2m}{2}, & \textrm{if $g$ is odd, $m=0, 1, \ldots, (g-3)/2$.}
\end{array} \right.
\]
Then $h$ is an integer such that $1 \leq h \leq g/2$. 
\begin{figure}[htb]
\labellist \small\hair 2pt
\pinlabel {$S_g$} [l] at 22 350
\pinlabel {$S_g$} [l] at 298 350
\pinlabel {$S_h$} [l] at 40 126
\pinlabel {$S_h$} [l] at 318 126
\pinlabel {$g$ even} at 114 10
\pinlabel {$g$ odd} at 393 13
\endlabellist
\centering
\includegraphics[scale=0.65]{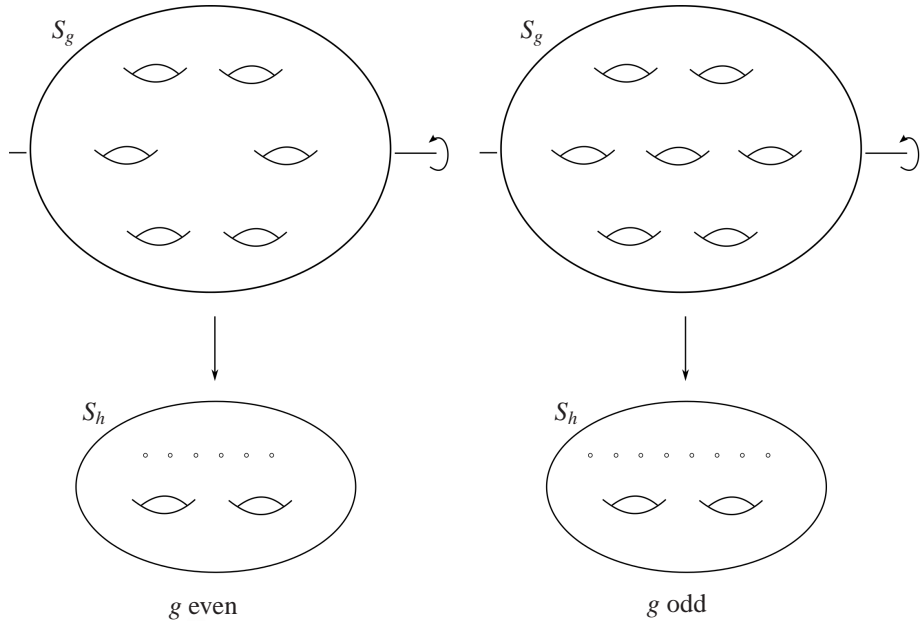}
\caption{A branched cover}
\label{figure:branched}
\end{figure}

Construct a branched cover $S_g \rightarrow S_h$ as in Figure \ref{figure:branched}. 
For $h \geq 2$, $S_h$ has a pseudo-Anosov mapping class $f_h \in \MCG(S_h)$ as in the Theorem \ref{thm:mainB} whose stretch factor has $\deg(\lambda(f_h)) = 2h$. For some $k$, ${f_{h}}^k$ lifts to $S_g$ and the lift has stretch factor $\lambda(f_h)^k$. We claim that $\deg\left( \lambda(f_h)^k \right) = 2h$. 
To see this, let $\lambda_i$, $1 \leq i \leq 2h$, be the roots of the minimal polynomial of $\lambda(f_h)$ and let us define a polynomial 
\[
p(x) = \prod_{i=1}^{2h} \left( x - \lambda_i^k \right).
\]
Then $p(x)$ is an integral polynomial because the following elementary symmetric polynomials in $\lambda_1, \ldots, \lambda_{2h}$ 
\[
\sum \lambda_i, \ \sum_{i<j} \lambda_i \lambda_j, \ \sum_{i<j<l} \lambda_i \lambda_j\lambda_l, \ \cdots, \ \lambda_1\lambda_2\cdots \lambda_{2h}
\]
are all integers and hence each coefficient of $p(x)$
\[
\sum \lambda_i^k, \ \sum_{i<j} \lambda_i^k \lambda_j^k, \ \sum_{i<j<l} \lambda_i^k \lambda_j^k \lambda_l^k, \ \cdots, \ \lambda_1^k\lambda_2^k\cdots \lambda_{2h}^k
\]
are integers as well. Therefore $p(x)$ is divided by the minimal polynomial of $\st(f_h)^k$. 
Due to the proof of Theorem \ref{thm:mainB} in section \ref{section:irreducibility},
$\lambda(f_h)^{k}$ is also a Salem number and $p(x)$ does not have a cyclotomic factor.
This implies that $p(x)$ is irreducible and $\deg\left( \lambda(f_h)^k \right) = 2h$.

If $h=1$, $S_h$ is a torus and it admits a Anosov mapping class $f$ whose stretch factor $\lambda(f)$ has algebraic degree $2$. Then similar arguments as above tells us that there is a lift of some power of $f$ to $S_g$ whose stretch factor has $\deg(\st(f^k))=2$.

Therefore there is a pseudo-Anosov map $\widetilde{f_h} \in \MCG(S_g)$ with $\deg(\lambda(\widetilde{f_h})) = 2h$ for each $h \leq g/2$.
In other words, every positive even degree $d \leq g$ is realized as the algebraic degree of a stretch factor on $S_g$.
\end{proof}

\medskip
\section{Stretch factors of odd degrees} \label{section:odd}
Long proved the following degree obstruction and McMullen communicated to us the following proof. First we will give a definition of the reciprocal polynomial. Given a polynomial $p(x)$ of degree $d$, we define the reciprocal polynomial $p^*(x)$ of $p(x)$ by $p^*(x)=x^d p(1/x)$. It is a well-known property that $p^*(x)$ is irreducible if and only if $p(x)$ is irreducible.
\begin{theorem}[\cite{long}]
Let $f \in \MCG(S_g)$ be a pseudo-Anosov mapping class having stretch factor $\lambda(f)$. If $\deg(\lambda(f)) > 3g-3$, then $\deg(\lambda(f))$ is even.
\end{theorem}
\begin{proof}
Since $f$ acts by a piecewise integral projective transformation on the $6g-6$ dimensional space $\PMF$ of projective measured foliations on $S_g$, and since $\lambda(f)$ is an eigenvalue of this action, $\st(f)$ is an algebraic integer with $\deg(\st(f)) \leq 6g-6$. Also, since $f$ preserves the symplectic structure on $\PMF$, it follows that $\st(f)$ is the root of palindromic polynomial $p(x)$ whose degree is bounded above by $6g-6$. 

Let $q(x)$ be the minimal polynomial of $\st(f)$ and let $q^{*}(x)$ be the reciprocal polynomial of $q(x)$. Then either $q(x) = q^{*}(x)$ or they have no common roots, because if there is at least one common root $\zeta$ of $q(x)$ and $q^{*}(x)$, then both $q(x)$ and $q^{*}(x)$ are the minimal polynomial of $\zeta$ and hence $q(x) = q^{*}(x)$. Suppose $\deg(q(x)) > 3g-3$. If $q(x)$ and $q^{*}(x)$ have no common roots, then their product $q(x) \, q^{*}(x)$ is a factor of $p(x)$ since $q^{*}(x)$ is the minimal polynomial of $1/\st(f)$. This is a contradiction because $\deg(p(x)) \leq 6g-6$ but $\deg\big(q(x) \, q^{*}(x)\big) > 6g-6$. Therefore we must have $q(x) = q^{*}(x)$ and this implies that $q(x)$ is an irreducible palindromic polynomial. Hence $\deg(q(x))$ is even since roots of $q(x)$ come in pairs, $\lambda_i$ and $1/\lambda_i$.
\end{proof}

It follows from the previous proof that if the minimal polynomial $p(x)$ of $\st$ has odd degree, then $p(x)$ is not palindromic and in fact the minimal palindromic polynomial containing $\st$ as a root is $p(x) p^*(x)$.

\medskip
We will now show that the stretch factors of degree 3 have an additional special property. A \textit{Pisot number}, also called a \textit{Pisot--Vijayaraghavan number} or a \textit{PV number}, is an algebraic integer greater than 1 such that all its Galois conjugates are strictly less than 1 in absolute value.

\begin{prop}
\label{thm:pisot}
Let $f \in \MCG(S_g)$. If $\deg(\st(f)) =3$, then $\st(f)$ is a Pisot number.
\end{prop}
\begin{proof}
Let $\lambda_1>1$ be the stretch factor of a pseudo-Anosov mapping class with algebraic degree 3, and let $p(x)$ be the minimal polynomial of $\lambda_1$. Let $\lambda_1, \lambda_2,$ and $\lambda_3$ be the roots of $p(x)$. Then the degree of $p(x) p^*(x)$ is 6 and it has pairs of roots $(\lambda_1, 1/\lambda_1), (\lambda_2, 1/\lambda_2), (\lambda_3, 1/\lambda_3)$, where $\lambda_1$ is the largest root in absolute value. We claim that the absolute values of $\lambda_2$ and $\lambda_3$ are strictly less than 1.

Suppose one of them has absolute value greater than or equal to 1, say $|\lambda_2| \geq 1$. The constant term $\lambda_1 \lambda_2 \lambda_3$ of $p(x)$ is $\pm 1$ since it is the factor of a palindromic polynomial with constant term 1. Hence $|\lambda_1 \lambda_2 \lambda_3|=1$ and we have
\[
\frac{1}{|\lambda_3|} = |\lambda_1 \lambda_2| \geq  |\lambda_1|,
\]
which is a contradiction to the fact that the stretch factor $\lambda_1$ is strictly greater than all other roots of the palindromic polynomial $p(x) p^*(x)$. This proves the claim and hence the stretch factor of degree 3 is a Pisot number.
\end{proof}
\medskip
We now explain two constructions of mapping classes $f \in \MCG(S_g)$ whose degree of $\st(f)$ is odd.

1. As we mentioned, Arnoux--Yoccoz \cite{arnouxyoccoz}  gave examples of a pseudo-Anosov mapping class on $S_g$ whose stretch factor has algebraic degree $g$. In particular for odd $g$, this gives examples of mapping classes with odd degree stretch factors. They proved that these stretch factors are all Pisot numbers.

2. For genus $2$, there is a pseudo-Anosov mapping class $f$ whose stretch factor has algebraic degree 3 (see section \ref{section:examples}). This is the only possible odd degree on $S_2$ by Long's obstruction. It is also true that $\deg (\lambda(f)^k) = 3$ for each $k$ because the stretch factor is a Pisot number (Proposition \ref{thm:pisot}). There is a cover $S_g \rightarrow S_2$ for each $g$, so the lift of some power of $f$ has a stretch factor with algebraic degree 3 on $S_g$.

\begin{prop}
For each genus $g$, the stretch factor with algebraic degree 3 can occur on $S_g$.
\end{prop}

\begin{ques}
Are there stretch factors with odd algebraic degree that are not Pisot numbers?
\end{ques}

\medskip
\section{Examples of even degrees} \label{section:examples}
Tables 1 through 4 give explicit examples of pseudo-Anosov mapping classes whose stretch factors realize various degrees. We will follow the notation of the software $\texttt{Xtrain}$ by Brinkmann. More specifically, $a_i, b_i, c_i,$ and $d_i$ are Dehn twists along standard curves and $A_i, B_i, C_i,$ and $D_i$ are the inverse twists as in \cite{brinkmann}. The only missing degree on $S_3$ is degree 5. We do not know if there is a degree 5 example or there is another degree obstruction.
{\small
\begin{table}[h]
\begin{center}
\renewcommand{\arraystretch}{1.2}
\begin{tabular}[c]{|c|l|l|c|}
\hline
deg & $\hspace{0.6em} f \in \MCG(S_2)$ & Minimal polynomial & $\lambda(f)$\\
\hline
2 & $a_0a_0d_0C_0D_1C_0$ & $x^2 - 3x + 1$ & $ \lambda =2.618$ \\
\hline
3 & $a_0d_0d_0C_0C_0D_1$ & $x^3 - 3x^2 -x -1$ & $ \lambda = 3.383$ \\
\hline
4 & $a_0d_0d_0d_1c_0d_0$ & $ x^4 - x^3 - x^2 -x + 1 $ & $ \lambda = 1.722$ \\
\hline
6 & $a_0a_0d_0A_0C_0D_1$ & $x^6 - x^5 - 4x^3 -x +1$ & $ \lambda = 2.015$ \\
\hline
\end{tabular}
\vspace{0.5em}
\caption{Examples of genus 2}
\end{center}
\end{table}
}

\vspace{0em}
\begin{table}[h]
{\small
\begin{center}
\renewcommand{\arraystretch}{1.2}
\begin{tabular}[c]{|c|l|l|c|}
\hline
deg & $\hspace{1em} f \in \MCG(S_3)$ & \hspace{6em} Minimal polynomial & $\lambda(f)$\\
\hline
2 & $a_1c_0d_0c_0d_2C_1D_1$ & {\footnotesize$x^2 - 4x + 1$} & $ 3.732$ \\
\hline
3 & $a_0c_0d_0C_1D_1D_2$ & {\footnotesize$x^3-2x^2+x-1$} & $  1.755$ \\
\hline
4 & $a_1c_0d_0a_1c_1d_1d_2$ & {\footnotesize$ x^4 - x^3 - 2x^2 - x + 1 $} & $ 1.722$ \\
\hline
6 & $a_0c_0d_0d_2C_1D_1$ & {\footnotesize$x^6 - 3x^5 + 3x^4 - 7x^3 + 3x^2 - 3x + 1$} & $ 2.739$ \\
\hline
8 & $a_0c_0d_0d_1C_1D_2$ & {\footnotesize$x^8-x^7-2x^5-2x^3-x+1$} & $ 1.809$ \\
\hline
10 & $a_1c_0d_0d_1C_1A_2D_2$ & {\footnotesize$x^{10} - x^9 - 2x^8 + 2x^7 - 2x^5 + 2x^3 - 2x^2 - x + 1$} & $  1.697$ \\
\hline
12 & $a_1c_1c_0d_1d_2A_0D_0$ & {\footnotesize $x^{12} - x^{11} - x^9 -x^8 + x^7 + x^5 - x^4 -x^3 -x +1$} & $  1.533$ \\
\hline
\end{tabular}
\vspace{0.5em}
\caption{Examples of genus 3}
\end{center}
}
\end{table}

\begin{table}[h]
{\small
\begin{center}
\renewcommand{\arraystretch}{1.2}
\begin{tabular}[c]{|c|l||c|l|}
\hline
deg & $\hspace{1.6em} f \in \MCG(S_4)$ & deg & $\hspace{1.6em} f \in \MCG(S_4)$ \\
\hline
4 & $a_0a_0a_1c_0d_0c_1d_1c_2d_2c_3d_3$ & 12 & $a_0B_1d_0c_0d_1c_1d_2c_2d_3c_3$\\
\hline
6 & $a_0B_2A_3d_0c_0d_1c_1d_2c_2d_3c_3$ & 14 & $a_0d_0B_0d_0c_0d_1c_1d_2c_2d_3c_3$ \\
\hline
8 & $a_0A_1d_0c_0d_1c_1d_2c_2d_3c_3$ & 16 & $A_0d_0c_0d_1c_1d_2c_2d_3c_3$ \\
\hline
10 & $a_0b_1A_2d_0c_0d_1c_1d_2c_2d_3c_3$ & 18 & $a_0B_1A_2d_0c_0d_1c_1d_2c_2d_3c_3$\\
\hline
\end{tabular}
\vspace{0.5em}
\caption{Examples of genus 4}
\end{center}}
\end{table}

\begin{table}[h!]
{\small
\begin{center}
\renewcommand{\arraystretch}{1.2}
\begin{tabular}[c]{|c|l||c|l|}
\hline
deg & $\hspace{2em} f \in \MCG(S_5)$ & deg & $\hspace{2em} f \in \MCG(S_5)$ \\
\hline
6 & $b_3d_0c_0d_1c_1d_2c_2d_3c_3d_4c_4$ & 16 & $a_1B_2d_0c_0d_1c_1d_2c_2d_3c_3d_4c_4$\\
\hline
8 & $a_0a_1d_0c_0d_1c_1d_2c_2d_3c_3d_4c_4$ & 18 & $a_1B_0d_0c_0d_1c_1d_2c_2d_3c_3d_4c_4$\\
\hline
10 & $a_1A_4d_0c_0d_1c_1d_2c_2d_3c_3d_4c_4$ & 20 & $a_1A_0d_0c_0d_1c_1d_2c_2d_3c_3d_4c_4$\\
\hline
12 & $b_2C_2d_0c_0d_1c_1d_2c_2d_3c_3d_4c_4$ & 22 & $a_2A_1d_0c_0d_1c_1d_2c_2d_3c_3d_4c_4$\\
\hline
14 & $a_1B_1d_0c_0d_1c_1d_2c_2d_3c_3d_4c_4$ & 24 & $c_2A_2d_0c_0d_1c_1d_2c_2d_3c_3d_4c_4$ \\
\hline
\end{tabular}
\vspace{0.5em}
\caption{Examples of genus 5}
\end{center}}
\end{table}

\medskip
\section{Irreducibility of Polynomials} \label{section:irreducibility}
In this section, we will prove Theorem \ref{thm:mainB}.
It is enough to show that the polynomial 
\[
p_n(x) = x^{2n} - 2\left( \sum_{j=1}^{2n-1} x^j \right) + 1.
\]
is irreiducible for $n \geq 2$.
We will show that $p_n(x)$ does not have a cyclotomic polynomial factor. It then follows from Kronecker's theorem that $p_n(x)$ is irreducible.

Suppose $p_n(x)$ has the $m$th cyclotomic polynomial factor for some $m \in \N$.
Then $e^{2\pi i/m}$ is a root of $p_n(x)$.
Multiplying $p_n(x)$ by $x-1$ yields
$$x^{2n+1} -3x^{2n} +3x -1$$
and hence we have 
\begin{equation}
\label{eqn:cyclo}
\ e^{2(2n+1)\pi i /m} -3e^{4n\pi i /m}+3e^{2\pi i /m}-1=0.
\end{equation}
Consider the real part and the complex part of (\ref{eqn:cyclo}). Then we have the system of equations
\[
 \left\{ \begin{array}{l}
         \cos\big(\frac{2(2n+1)\pi}{m}\big) -3\cos\big(\frac{4n\pi}{m}\big) +3\cos\big(\frac{2\pi}{m}\big)-1=0\\
	 \sin\big(\frac{2(2n+1)\pi}{m}\big) -3\sin\big(\frac{4n\pi}{m}\big) +3\sin\big(\frac{2\pi}{m}\big)=0\\
        \end{array} \right.
\]
Using double-angle formula for the first cosine and sum-to-product formula for the last two cosines, the first equation gives
\[
2\sin\Big(\frac{(2n+1)\pi}{m}\Big) \Big[ 3\sin\Big(\frac{(2n-1)\pi}{m}\Big) -\sin\Big(\frac{(2n+1)\pi}{m}\Big)\Big] =0.
\]
Similarly the second equation gives
\[
 2\cos\Big(\frac{(2n+1)\pi}{m}\Big) \Big[ \sin\Big(\frac{(2n+1)\pi}{m}\Big) - 3\sin\Big(\frac{(2n-1)\pi}{m}\Big) \Big] =0.
\]
Since sine and cosine have no common zeros, we must have
\begin{equation*}
 \sin\Big(\frac{(2n+1)\pi}{m}\Big) - 3\sin\Big(\frac{(2n-1)\pi}{m}\Big)=0.
\end{equation*}
For $m \leq 5$, by direct calculation we can see that $p_n(e^{2\pi i/m}) \neq 0$. So we may assume that $m \geq 6$. 
Let $\varphi = (2n-1)\pi/m$ and then we can write the above equation as
\begin{equation}
\label{eqn:sin}
 \sin\Big(\varphi + \frac{2\pi}{m}\Big) - 3\sin(\varphi)=0.
\end{equation}
Since $\sin\big( \varphi + 2\pi/m \big)$ is a real number between $-1$ and $1$, we have
\begin{equation}
\label{eqn:range}
  -\frac{1}{3} \leq \sin(\varphi) \leq \frac{1}{3}.
\end{equation}
Let $\psi = \sin^{-1}(1/3)$. Then note that $\psi < \pi/6$. 
Equation (\ref{eqn:range}) gives the restriction on $\varphi$, which is
\[
 -\psi \leq \varphi \leq \psi \ \ \textrm{or} \ \ \pi -\psi \leq \varphi \leq \pi + \psi.
\]
Another observation from (\ref{eqn:sin}) is that both $\sin\big( \varphi + 2\pi/m \big)$ and $\sin(\varphi)$ must have the same sign.

We claim that $\varphi$ has to be on the either first or third quadrant. Suppose $\varphi$ is on the second quadrant, that is, $ \pi -\psi < \varphi < \pi$. Note that $m \geq 6$ implies $ 2\pi/m \leq \pi/3$. Since $\varphi$ is above the $x$-axis, $\varphi + 2\pi/m$ also has to be above the $x$-axis due to (\ref{eqn:sin}) and hence the only possibility is that $\varphi + 2\pi/m$ is between $\varphi$ and $\pi$. Then
\[
  0<\sin\Big(\varphi + \frac{2\pi}{m}\Big) < \sin(\varphi) \ \implies \  \sin\Big(\varphi + \frac{2\pi}{m}\Big) < 3\sin(\varphi),
\]
which is a contradiction to (\ref{eqn:sin}). Similar arguments hold if $\varphi$ is on the fourth quadrant. Therefore the possible range for $\varphi$ is
\[
 0<\varphi \leq \psi \ \ \textrm{or} \ \ \pi < \varphi \leq \pi+\psi.
\]

Suppose $\varphi$ is on the first quadrant. Then so is $\varphi + 2\pi/m$ because
\[
 0<\varphi + \frac{2\pi}{m} \leq \psi + \frac{\pi}{3} < \frac{\pi}{2}.
\]
We can write
\[
 \varphi = \frac{(2n-1)\pi}{m} \equiv \ \frac{j\pi}{m} \pmod{2\pi}
\]
for some positive integer $j$, i.e., $0 < j\pi/m < \pi/2$.

If $j\geq2$, Using the subadditivity of $\sin(x)$ on the first quadrant
\[
 \sin(x+y) \leq \sin(x) + \sin(y),
\]
we have
{\setlength\arraycolsep{2pt}
\begin{eqnarray*}
 \sin\Big(\varphi + \frac{2\pi}{m}\Big) - 3\sin(\varphi) &\leq& \left( \sin(\varphi) + \sin\Big(\frac{2\pi}{m}\Big) \right) - 3\sin(\varphi)\\
 &=&  \sin\Big(\frac{2\pi}{m}\Big) -2\sin(\varphi)\\
&=& \sin\Big(\frac{2\pi}{m}\Big) - 2\sin\Big(\frac{j\pi}{m} \Big) < 0,
\end{eqnarray*}
}
which contradicts (\ref{eqn:sin}).

If $j=1$, using triple-angle formula
{\setlength\arraycolsep{2pt}
\begin{eqnarray*}
 \sin\Big(\varphi + \frac{2\pi}{m}\Big) - 3\sin(\varphi) &=& \sin\Big(\frac{3\pi}{m}\Big) - 3\sin\Big(\frac{\pi}{m} \Big) \\
&=& \left( 3\sin\Big(\frac{\pi}{m}\Big) - 4 \sin^3\Big(\frac{\pi}{m}\Big) \right) - 3\sin\Big(\frac{\pi}{m}\Big)\\
&=& - 4 \sin^3\Big(\frac{\pi}{m}\Big) <0,
\end{eqnarray*}
}
which contradicts (\ref{eqn:sin}) again. Therefore there is no possible $\varphi$ on the first quadrant. By using the same arguments, the fact that $\varphi$ is on the third quadrant gives a contradiction. Therefore we can conclude that $p(x)$ does not have a cyclotomic factor.

We now show that $p_n(x)$ is irreducible over $\Z$. Suppose $p_n(x)$ is reducible and write $p_n(x)=g(x)h(x)$ with non-constant functions $g(x)$ and $h(x)$. There is only one root of $p_n(x)$ whose absolute value is strictly greater than 1. Therefore one of $g(x)$ or $h(x)$ has all roots inside the unit disk. By Kronecker's theorem, this polynomial has to be a product of cyclotomic polynomials, which is a contradiction because $p_n(x)$ does not have a cyclotomic polynomial factor. Therefore $p_n(x)$ is irreducible.

\bibliographystyle{plain}
\bibliography{degree2g}

\begin{thebibliography}{10}

\bibitem{arnouxyoccoz}
Pierre Arnoux and Jean-Christophe Yoccoz.
\newblock Construction de diff\'eomorphismes pseudo-{A}nosov.
\newblock {\em C. R. Acad. Sci. Paris S\'er. I Math.}, 292(1):75--78, 1981.

\bibitem{brinkmann}
Peter Brinkmann.
\newblock An implementation of the {B}estvina-{H}andel algorithm for surface
  homeomorphisms.
\newblock {\em Experiment. Math.}, 9(2):235--240, 2000.

\bibitem{primer}
Benson Farb and Dan Margalit.
\newblock {\em A primer on mapping class groups}, volume~49 of {\em Princeton
  Mathematical Series}.
\newblock Princeton University Press, Princeton, NJ, 2012.

\bibitem{fried}
David Fried.
\newblock Growth rate of surface homeomorphisms and flow equivalence.
\newblock {\em Ergodic Theory Dynam. Systems}, 5(4):539--563, 1985.

\bibitem{minimal}
Erwan Lanneau and Jean-Luc Thiffeault.
\newblock On the minimum dilatation of pseudo-{A}nosov homeromorphisms on
  surfaces of small genus.
\newblock {\em Ann. Inst. Fourier (Grenoble)}, 61(1):105--144, 2011.

\bibitem{Leininger04}
Christopher~J. Leininger.
\newblock On groups generated by two positive multi-twists: {T}eichm\"uller
  curves and {L}ehmer's number.
\newblock {\em Geom. Topol.}, 8:1301--1359 (electronic), 2004.

\bibitem{LepovicGutman01}
M.~Lepovi{\'c} and I.~Gutman.
\newblock Some spectral properties of starlike trees.
\newblock {\em Bull. Cl. Sci. Math. Nat. Sci. Math.}, (26):107--113, 2001.
\newblock The 100th anniversary of the birthday of Academician Jovan Karamata.

\bibitem{long}
D.~D. Long.
\newblock Constructing pseudo-{A}nosov maps.
\newblock In {\em Knot theory and manifolds ({V}ancouver, {B}.{C}., 1983)},
  volume 1144 of {\em Lecture Notes in Math.}, pages 108--114. Springer,
  Berlin, 1985.

\bibitem{McKeeRowlinsonSmyth99}
J.~F. McKee, P.~Rowlinson, and C.~J. Smyth.
\newblock Salem numbers and {P}isot numbers from stars.
\newblock In {\em Number theory in progress, {V}ol. 1
  ({Z}akopane-{K}o\'scielisko, 1997)}, pages 309--319. de Gruyter, Berlin,
  1999.

\bibitem{neuwirthpatterson}
L.~Neuwirth and N.~Patterson.
\newblock A sequence of pseudo-{A}nosov diffeomorphisms.
\newblock In {\em Combinatorial group theory and topology ({A}lta, {U}tah,
  1984)}, volume 111 of {\em Ann. of Math. Stud.}, pages 443--449. Princeton
  Univ. Press, Princeton, NJ, 1987.

\bibitem{Shin15}
Hyunshik Shin.
\newblock Spectral radius of starlike trees with one long arm.
\newblock Preprint.

\bibitem{thurston}
William~P. Thurston.
\newblock On the geometry and dynamics of diffeomorphisms of surfaces.
\newblock {\em Bull. Amer. Math. Soc. (N.S.)}, 19(2):417--431, 1988.

\end{thebibliography}

\end{document}